\documentclass[10pt,amstex,amssymb]{amsart}
\usepackage{graphicx} 
\def\PP{\mathbb{P}}
\def\EE{\mathbb{E}}

\newtheorem{teo}{Theorem}[section]    
\newtheorem{lem}{Lemma}[section]   
\newtheorem{pro}{Proposition}[section]

\begin{document}

\begin{center}
{\sc \bf 
    SIR EPIDEMICS ON A SCALE-FREE  SPATIAL NESTED MODULAR NETWORK
     WITH NON-TRIVIAL THRESHOLD

}
\end{center}

\vskip .35in

\begin{center}

\vskip .25in

\vskip .15in
Alberto Gandolfi\\
Dipartimento di Matematica U. Dini,\\
 Universit\`{a} di
Firenze,\\ Viale Morgagni 67/A, 50134 Firenze, Italy \\ 
email: 
albertogandolfi@gmail.com\\
\vskip .25in

\vskip .15in
Lorenzo Cecconi\\
Dipartimento di Matematica U. Dini,\\
 Universit\`{a} di
Firenze,\\ Viale Morgagni 67/A, 50134 Firenze, Italy \\ 
email: 
cecconi@math.unifi.it\\
\vskip .25in

\end{center}

\vskip 0.2in

\begin{verse}
\hspace{.25in}{\bf \sc Abstract.}
{\footnotesize 
We propose a class of random scale-free spatial networks with nested
community structures and analyze  Reed-Frost epidemics
with  class related independent transmissions. We show 
that the epidemic threshold may be trivial or not depending
on the relation among community sizes, distribution of the number of communities
and transmission rates.}

\end{verse}



{\it Secondary} 93A13; 82B43; 92C60.

\bigskip

\section{Introduction}

We consider a spatial random graph which  at the same time  is scale-free
and has a nested community structure, and study Reed-Frost SIR epidemic (\cite{Kermacketal}, \cite{Neal})  on it. We 
find that with a natural transmission mechanism, in which
transmissions occur independently with rates related to community
sizes,  the critical threshold is trivial or not
depending on the relation between community sizes, distribution 
of number of communities to which each individual belongs and
rate of the decay of the transmission probability 
as the community size increases.
  Scale-free networks (\cite{BarabasiAlbert}, \cite{AlbertJeongBarabasi},
  \cite{EriksenHornquist}) have been widely studied in the context of epidemics
(see  \cite{SanderWarrenSokolov} and
\cite{BrittonJansonMartinLof}) suggesting at first that this might 
lead to triviality of the critical threshold (\cite{MorenoGomezPacheco},
\cite{KephartSorkinChess}, \cite{KephartWhiteChess}, \cite{ZhouFuWang}).
On the other hand, most scale-free networks lack a spatial dimensionality, 
which is quite relevant to make the models more realistic (see e.g. \cite{Britton}): one of the few prosed networks possessing both features has been 
suggested by Yukich \cite{Yukich}.
Yukich's network is, however, missing network modularity,
 i.e. the gathering of individuals in 
communities with
faster transmission rates 
 (see \cite{Bartoszynski}, \cite{BeckerDietz}, \cite{BallMollisonScalia-Tomba},  \cite{BallSirlTrapman1} and \cite{BallSirlTrapman2}),
a feature which has gained recent interest due to its relevance
in infectious transmission. The formation of communities can be described
by several mechanisms, such as random intersection  in which 
extra vertices randomly connect to the vertices of the
graph and links are then generated between vertices connected
to a common extra vertex (see \cite{BrittonDeijfenAndreasLagerasLindholm} for a description of
random intersection and a review of other mechanisms). However, most 
real community structures are
nested (see, e.g., \cite{StokolsClitheroe}, \cite{Tropmanetal} and \cite{Costelloetal}) unlike the networks generated by random intersection
and similar mechanisms.

The class of random networks discussed here have spatial features, are scale-free and possess a nested community structure.
The 
networks are based on a connectivity graph, which, for simplicity,  is here taken to be
$\mathbb Z^d$  endowed with a hierarchical structure of partitions
into larger and larger communities. To generate the network each
vertex $v \in \mathbb Z^d$ is assigned
a random integer value $X_v$, where the $X_v$'s are i.i.d. random
variables. Each vertex $v$ identifies an individual, which belongs to all communities up to level
$X_v$ in the hierarchical structure. The basic random connectivity graph
is obtained by adding to the nearest neighbor edges of $\mathbb Z^d$ all the edges between
pairs of vertices belonging to at least one common community. For 
a wide class of distributions of the $X_v$'s the connectivity
graph is scale-free.

We then consider  Reed-Frost SIR epidemics on the connectivity network, in 
which infected individuals at time $t$ contact each neighbor independently with some transmission probability, and
if the neighbor is susceptible it becomes infected.
To complete the model, it is natural to consider  basic transmission probabilities for nearest neighbor vertices,
and then an additional probability, decreasing with the size of the 
community, of independent transmission for 
any community shared by two individuals. In this way,
the transmission probabilities do not depend only from the connectivity
graph, but directly from the shared classes, and give rise to a very
realistic mechanism. The set of individuals
ultimately affected by 
the Reed-Frost SIR epidemic is the set of vertices belonging to
a percolation graph with connection probabilities given by
the transmission probabilities (\cite{KuulasmaaZachary}); for natural choices
of the probabilities of infection through shared communities
the phase diagram of the percolation graph
exhibits a transition from nontrivial to trivial percolation threshold. \\
In summary, the model depends on five parameters: 
\begin{itemize}
\item $d$, indicating space dimension;
\item $z$, determining the growth factor $z^d$ of community sizes;
\item $\alpha \geq 1$, determining the distribution of the number of communities to which an individual belongs;
\item $p$, indicating the transmission probabilities to neighbors;
\item $\rho$, modulating the decrease in transmission probabilities for large communities.
\end{itemize}
Several random networks can be generated along the indicated lines. In particular,
the construction must specify the form of each partition and the interconnections
between partitions. To illustrate the mathematical properties of the
networks, we discuss in Section 2 a very simple and schematic structure, in which 
at each level $k$ the space is partitioned into hypercubes of linear size $z^k$,
which are then packed into hypercubes of linear size $z^{k+1}$ and so on. To keep
things simple one can 
think to $z=2$. For simplicity, we also limit ourselves to just one
 single parameter $\alpha$ to generate the 
connectivity graph, although this is excessively simplified, as the 
inclusion in small communities is likely to follow a different pattern from that
of inclusion in large communities.
In Section 2 we give a detailed description of the construction of the connectivity network.

In Section 3 we show that the degree distribution $D_v$ of any vertex $v$ in the connectivity network satisfies 
$$\lim_{h \rightarrow \infty} P(D_v \geq h) \ h^{- \gamma + 1} \in \bigg[\frac{1}{2 \alpha},
\frac{d^{d/2 +1} \omega_d}{d-\log_z \alpha}\bigg]$$
where $\omega_d$ is the volume of the $d$-dimensional unit ball and
$\gamma -1= \frac{\log_z \alpha}{d-\log_z \alpha}$, so that  
the network is scale-free for all  $\alpha \in (1,z^d)$; in particular, 
for 
$z^{\frac{d}{2}} \leq \alpha \leq z^{\frac{2d}{3}}$ the
network exhibits the typical value of $\gamma \in (2,3)$. 

In Section 4 we complete the description of the Reed-Frost
epidemic and begin the description of the phase diagram in the $\alpha - \rho$ 
variables; such description is completed in Section 5 
by dominating the probability of transmission in certain
sets by those in a long range percolation model extending  a recent result in \cite{MeesterTrapman}; it is remarkable that although we 
use  edge variables to bound a model based on site variables the 
result is still sharp and we identify the exact phase diagram.

\begin{enumerate}
\item For $ \alpha \geq z^d$ the network has short range behavior, and the hierarchical communities structure is irrelevant for the existence of 
critical threshold: there is a critical epidemic threshold $p_c$ for all $\rho$. 
\item For $ 1 \leq \alpha < z^d$ the behavior depends on $\rho$: if $\rho < \frac{\alpha}{z^d}$
there is still a nontrivial epidemic critical threshold $p_c \in (0,1)$ 
while $p_c=0$ for $\rho > \frac{\alpha}{z^d}$. This means that percolation, and thus an infinite outbreak, occurs at all values of $p$ in the parameter range we just identified. In the scale-free region, determined by $\alpha \in (1, z^{d})$, $p_c$ is thus trivial or not depending on the transmission rate in large communities.
It is trivial if the transmission rates are constant ($\rho=1$) or with a not too fast rate; on the other hand it is not trivial for $\rho$ below a critical curve in the phase diagram.
\item On the line $\alpha=1$ each vertex belongs to all communities:
the model is similar to long-range percolation (see \cite{NewmanSchulman}) and 
is studied in \cite{MeesterTrapman2}: $p_c=0$ or $p_c>0$ for the same parameter range 
as in long-range percolation.
\end{enumerate}
In a sense the proposed model interpolates between short ($\alpha \geq z^d$) and long ($\alpha = 1 $) range percolation.
A summary of the phase diagram is in figure 1.

\section{The connectivity graph} 
Consider a random graph $G_{\alpha, z} = (\mathbb Z^d, E_{\alpha, z})$ with $\mathbb Z^d$ as set of vertices and a random set of edges $\mathbb E_{\alpha,z}$ to be specified.
In the first place, $\mathbb E_1^d \subseteq E_{\alpha,z}$, where $\mathbb E_1^d$ is the set of nearest neighbor 
edges of $\mathbb Z^d$.  
Then, consider i.i.d. random variables $X_v$, $v \in \mathbb Z^d$, with a nonnegative
integer distribution $\mu_{\alpha,v}$ such that
$\mu_{\alpha,v}(X_v \geq k) = \alpha^{-k}$, $k=0, 1, \dots$,
where $\alpha > 1$ is a parameter. We let 
$\mu_{\alpha}=\prod_{v \in \mathbb Z^d} \mu_{\alpha}$ the joint product distribution 
of the $X_v$'s on the Borel $\sigma$-algebra in 
$X= \mathbb N^{\mathbb Z^d}$. By this choice there is only one
parameter determining the distribution of the number of communities to which
 individuals belong; the average number of communities
to which an individual belongs, a measure called group membership 
(see \cite{GlaeserLaibsonSacerdote} and \cite{Newman2}), is $\sum_k \alpha^{-d}=(\alpha-1)^{-1}$.
This is a realistic number especially for $\alpha \in [5/4,2)$.
 
\noindent Next, let $z \geq 2$ be an integer and for each $k$ partition $ \mathbb Z^d$ into blocks 
$$ B_{z,k}(i)= B_{z,k}(i_1, \dots, i_d)=\{v=(v_1, \dots, v_d) \in \mathbb Z^d: z^k i_j \leq v_j \leq (i_j+1)z^k -1,
\text{ for all } j=1, \dots, d \}.$$
Blocks represent a system of nested communities. Note that vertices separated by coordinate hyperplanes lie
always in different communities; the community structure is 
thus confined to orthants, and vertices in different
orthants are connected only through nearest
neighbor connections: this is not an unrealistic feature, however,
as it might represent very rigid borders or seas.
 
Given $\rho_{\alpha}$ and the $B_{z,k}(i)$'s, the random connectivity 
graph $G_{\alpha, z}$ is completed by including into the
edge set $E_{\alpha, z}$, next to the nearest neighbor edges,  also 
all pairs $\{u,v\}$ such that
$\exists k \in \mathbb N, i \in \mathbb Z^d \textrm{ with } X_u,X_v \geq k \textrm{ and } u,v \in B_{z,k}(i).$
In other words, given $\alpha$, $z$, $\mu_{\alpha}$ and $B_{z,k}(i)$'s, the random graph $G_{\alpha, z}$ is
defined by a map $\phi_z: X \rightarrow H=\{0,1\}^{\mathbb E^d} $,
where $\mathbb E^d= \{ \{u,v\}: u,v \in \mathbb Z^d\}$, with
$$
\phi_z(x)_{\{u,v\}}= 
\mathbb I_{ \{ \exists k \in \mathbb N, \ i \in  \mathbb Z^d | x_u,x_v \geq k \textrm{ and } u,v \in B_{z,k}(i) \}} \vee 
\mathbb I_{ \{ |u -v|=1 \} } 
$$
by $P_{\alpha, z}= \mu_{\alpha}(\phi_z^{-1})$.
Later we are interested not only in the connectivity graph but also
in the set of communities joining each pair of vertices: this leads
to further specify the map $\phi_z$, as done in section $4$ below,
but we first study the connectivity properties of $G_{\alpha,z}$.

For $v \in \mathbb Z^d$ and  $\eta \in H$, let
$D_v=D_v (\eta)=|\{u \in \mathbb Z^d : \{u,v\} \in \mathbb E_{\alpha,z}\}|$
be the degree of $v$.

\begin{lem} \label{teo:lowerbound}
For all $v \in \mathbb Z^d$ and $\alpha$, $z$ such that $\alpha < z^d$
\begin{equation}
\lim_{h \rightarrow \infty} \PP_{\alpha, z}(D_v \geq h) \cdot h^{ \gamma-1} \geq \frac{1}{2 \alpha},
\end{equation} 
where 
$$ \gamma-1=\frac{\log_z \alpha}{d- \log_z \alpha}$$
\end{lem}

\begin{proof} Given $v \in \mathbb Z^d$ and $h \in \mathbb N$
consider  the block $B_{z,l(h)}=B_{z,l(h)}(i)$ such that $v \in B_{z,l(h)}$
with $l(h) = \lfloor \frac{1}{d- \log_z \alpha}\log_z h +1 \rfloor$.
Then
\begin{eqnarray*}
\EE_{\mu_{\alpha}} \bigg( \sum_{u \in B_{z,l(h)}} \mathbb I_{ \{X_u \geq l(h)\} } \bigg) & = &
z^{l(h) d } \mu_{\alpha}(X_u \geq l(h))  \\
& = & \Big(\frac{z^d}{ \alpha}\Big)^{l(h)} \ \geq \ \Big(\frac{z^d}{ \alpha}\Big)^{\frac{\log_z \alpha}{d- \log_z \alpha}} \ = \ h.
\end{eqnarray*}
By the CLT, $\lim_h \mu_{\alpha} (  \sum_{u \in B_{z,l(h)}} \mathbb I_{\{X_u \geq l(h)\}} \geq h ) \geq 1/2$.
Hence,
\begin{eqnarray*} 
h^{ \gamma -1} \ \PP_{ \alpha, z} (D_v \geq h) & \geq & h^{ \gamma -1} \ \mu_{\alpha}(X_v \geq l(h)) \
\PP_{ \alpha, z} \bigg( \sum_{u \in B_{z,l(h)}} \mathbb I_{ \{ \{v,u\} \in E_{\alpha, d} \}} \geq h \ | \
 X_v \geq l(h) \bigg) \\
& \geq & h^{\gamma -1} \ \alpha^{-l(h)} \ 
\mu_{\alpha} \bigg( \sum_{u \in B_{z,l(h)}} \mathbb I_{ \{X_u \geq l(h)\}} \geq h \bigg) 
\rightarrow \frac{1}{2 \alpha}
\end{eqnarray*}
\end{proof}

To get the corresponding upper bound on the degree distribution we compare 
the connectivity graph to Yukich's network, which has vertex heights
based on 
uniform distributions and connections related to distance. As first step, we
compare the connectivity graph  to a network based only on distances
but retaining the distribution of the $Z_v$'s for the vertex heights: for 
$\delta >0$
 consider $G'_{\alpha, z, \delta} = (\mathbb Z^d, E'_{\alpha, z, \delta})$
such that 
\begin{equation}
(u,v) \in E'_{\alpha, z, \delta} \Leftrightarrow \exists k \textrm{ s.t. } X_u,X_v \geq k \textrm{ and } d(u,v) \leq \delta z^k;
\end{equation}
more precisely, let $\phi_{z,\delta}' :X \rightarrow H$ such that
\begin{equation}
\phi_{z,\delta}'(x)_{\{u,v\}} =  \mathbb I_{ \{ \exists k \in \mathbb N \ | \ X_u,X_v \geq k \textrm{ and } d(u,v) \leq \delta z^k \}}(x)
\end{equation}
and let $P_{\alpha, z, \delta}'= \mu_{\alpha}((\phi'_{z,\delta})^{-1})$.
Note that for $\delta = \sqrt d$ and for every increasing $A \subset H$
\begin{equation}
\PP_{z,\alpha}(A) \leq \PP_{\alpha,z,\sqrt d}'(A).
\end{equation}
In fact, taking $k=\min(X_u,X_v)$, if $d(u,v) \geq \sqrt{d} \ z^k$ then $\{u,v\} \notin E_{\alpha, z}$. Therefore, if $\{u,v\} \in E_{\alpha,z}$ then $z^k \geq d(u,v) / \sqrt{d}$, so that $\{u,v\} \in E_{\alpha, z, \sqrt{d}}'$, i.e.
$\phi_{z}(x) \leq \phi'_{z, \sqrt{d}}(x)$. This implies that if $A$ is increasing, $x \in X$ and $\phi_z(x) \in A$
then also $\phi'_{z, \sqrt{d}}(x) \in A$, i.e.
$(\phi_{z})^{-1}(A) \subseteq \Big(\phi'_{z, \sqrt{d}}\Big)^{-1}(A).$

\noindent Note also that 
\begin{equation}
\phi_{z,\delta}'(x)_{\{u,v\}} = \mathbb I_{ \{z^{X_u},z^{X_v} \geq \frac{d(u,v)}{\delta}  \}}(x).
\end{equation}

\section{Comparison with Yukich network}
Let $\{U_v \}_{v \in \mathbb Z^d}$ be i.i.d. uniform $[0,1]$ random variables
with distribution $P_U$ on $[0,1]^{\mathbb Z^d}$ and consider  Yukich network $\bar G_{s,\delta}=(\mathbb Z^d,\bar E_{s,\delta})$ defined for $s,\delta>0$ by
\begin{equation}
\{u,v\} \in \bar E_{s,\delta} \ \Leftrightarrow \ d(u,v) \leq \delta \min(U_u^{-s},U_v^{-s})
\end{equation}
As before one can take 
$W=[0,1]^{\mathbb Z^d}$, define $\bar \phi'_{s,\delta}:W \rightarrow H$ such that
\begin{equation}
\bar \phi'_{s,\delta}(w)_{\{u,v\}} = \mathbb I_{ \{w_u,w_v \leq \frac{d(u,v)}{\delta}^{-1/s}  \}}
\end{equation}
and let $\mathbb {\bar P}'_{s,\delta}=P_U((\bar \phi'_{s,\delta})^{-1})$.
We need to slightly reformulate Theorem 1.1 in Yukich (\cite{Yukich})  to incorporate the constant $\delta$.
\begin{pro} \label{teo:yukich}
For all $d, \delta$ and $s \in (\frac{1}{d},\infty)$
\begin{equation}
\lim_{t \rightarrow \infty} t^{\frac{1}{sd-1}} \ \bar \PP'_{s,\delta}(D_s(v) \geq t) = \bigg( \frac{\delta^d s d \omega_d}{sd-1} \bigg)^{\frac{1}{sd-1}}
\end{equation}
for all $v \in \mathbb Z^d$, where $\omega_d$ denotes the volume of the unit ball in $\mathbb R^d$.
\end{pro}
\begin{proof}
Yukich proves the same result for $\delta=1$. The conclusion is achieved by taking the origin $v=0$, conditioning on $U_0=\tau$ and using translation invariance. The basis of Yukich proof is Lemma 2.1 in \cite{Yukich}, which states that
$$\EE(D_s(0)|U_0 =\tau) \approx \int_{|x|\leq \tau^{-s}} |x|^{-\frac{1}{s}} \ dx = d \omega_d \int_0^{\tau^{-s}} t^{d-1-\frac{1}{s}} \ dt = \frac{sd\omega_d}{sd-1} \ \tau^{-(sd-1)}$$
When a generic $\delta$ is considered we get
\begin{eqnarray*}
\EE(D_s(0)|U_0 =\tau) & \approx & \int_{|x|\leq \delta \tau^{-s}} |x|^{-\frac{1}{s}} \delta^{\frac{1}{s}} \ dx \\
& = & d \omega_d \delta^{\frac{1}{s}} \int_0^{\delta \tau^{-s}} t^{d-1-\frac{1}{s}} \ dt = \frac{\delta^d sd \omega_d}{sd-1} \ \tau^{-(sd-1)}
\end{eqnarray*}
The rest of the proof in \cite{Yukich}
is still valid with the constant $\beta=\frac{sd \omega_d}{sd-1}$ replaced by $\frac{\delta^d sd \omega_d}{sd-1}$
\end{proof}
\noindent Then we can deduce the following upper bound for the power law distribution of the network $G_{\alpha,z}$.

\begin{teo}\label{teo:upperbound}
For all $z$ and $\alpha \in (1,z^d)$
\begin{equation}\label{eq:upperbound}
\lim_{h \rightarrow \infty} \PP_{\alpha,z}(D(v) \geq h) \ h^{\gamma-1} \leq 
\bigg( \frac{d^{\frac{d}{2}+1} \omega_d}{d-\log_z \alpha} \bigg)^{\gamma-1}
\end{equation}
where $\gamma-1 = \frac{\log_z \alpha}{d-\log_z \alpha}$
\end{teo}
\begin{proof}
First recall that 
\begin{equation}\label{eq:dominationP2}
\PP_{z,\alpha}(D \geq h) \leq \PP_{z,\alpha,\sqrt{d}}'(D \geq h)
\end{equation}
since $\{D \geq h\}$ is increasing.
We want to compare $G_{\alpha,z,\delta}'$ to the Yukich's network $\bar G'_{s,\delta}$. 
For $m=z^k$
$$
\mu_{\alpha}(z^{X_v} \geq m)= \mu_{\alpha}(X_v \geq	k) = \alpha^{-k}
= \alpha^{-\log_z m}=m^{-\log_z \alpha}.
$$
On the other hand $P_U(U_v^{-s} \geq m)=P_U(U_v \leq m^{-1/s}) = m ^{-1/s}$
so that taking $\log_z \alpha= 1/s$ we have
$$
P_U(U_v^{-s} \geq m)=\mu_{\alpha}(z^{X_v} \geq m) \quad \text{ for } m=z^k
$$
and
$$
P_U(U_v^{-s} \geq m) = \alpha^{-\log_z m}
\geq \alpha^{\lceil  -\log_z m \rceil }= 
\mu_{\alpha}(X_v \geq \lceil  \log_z m \rceil)
=\mu_{\alpha}(X_v \geq   \log_z m )
=\mu_{\alpha}(z^{X_v} \geq	m)
$$
for all other $m \in \mathbb R$.
Therefore, $U_v$ and $z^{X_v}$ can be coupled by the following joint
distribution. Let $P_{U_v}$ be the distribution of
$U_v$ and $\nu_v$ be a probability on the $\sigma$-algebra in
$[0,1] \times \mathbb N$ such that for $A \subseteq [0,1]$ and 
$k \in \mathbb N$ it holds
$$
\nu_v(A, k)= P_{U_v}\Big(A \cap (z^{-(k+1)/s},z^{-k/s}]\Big).
$$
We have
\begin{itemize}
\item $\nu_v(A, \mathbb N)= P_{U_v}(A)$;
\item $\nu_v([0,1], k)= P_{U_v}((z^{-(k+1)/s},z^{-(k)/s}])
=\alpha^{-k} - \alpha^{-(k+1)}= \mu_{\alpha}(X_v=k);$
\item $\nu_v\{(u,k): u^{-s} \geq z^k\}
=\nu_v\{(u,k): u \leq z^{-k/s}\}=1$.
\end{itemize}
The product distributions $\mu_{\alpha}$ and 
$P_U$ can be coupled by the product distribution
$\nu= \prod_{v \in \mathbb Z^d} \nu_v$, under which 
$U_v^{-s} \geq z^{X_v}$ for all $v \in \mathbb Z^d$ with probability one. \\
If $w \in W$ and $x \in \mathbb N^{\mathbb Z^d}$ are such that
$w_v^{-s} \geq z^{x_v}$ for all $v$ then
\begin{eqnarray*}
\bar \phi'_{s,\delta}(w)_{ \{u,v\}} & = & \mathbb I_{ \{ w_u^{-1},w_v^{-1} \geq \frac{d(u,v)}{\delta} \} } \\
&\geq& \mathbb I_{ \{ z^{x_u},z^{x_v}\geq \frac{d(u,v)}{\delta} \} } = \phi'_{z,\delta}(x)_{ \{u,v\}}. 
\end{eqnarray*}
Thus,  if $A \subseteq H$ is increasing,  
$x \in (\phi'_{z,\delta})^{-1}(A)$ and $w_v^{-s} \geq z^{x_v}$
then $w \in (\bar \phi'_{s,\delta})^{-1}(A)$.
Hence, for $A$ increasing
\begin{eqnarray*}
P'_{\alpha, z, \delta}(A) &=& \mu_{\alpha} \Big( (\phi'_{z,\delta})^{-1}(A) \Big) \\
&=& \nu \Big( W,(\phi'_{z,\delta})^{-1}(A) \Big) \\
&=& \nu \Big( (\bar \phi'_{(\log_z \alpha)^{-1},\delta})^{-1}(A),(\phi'_{z,\delta})^{-1}(A) \Big)  \\
& \leq & \nu \Big( (\bar \phi'_{(\log_z \alpha)^{-1},\delta})^{-1}(A),\mathbb N^{\mathbb Z^d} \Big) 
= P_U \Big( (\bar \phi'_{(\log_z \alpha)^{-1},\delta})^{-1}(A) \Big)
= \bar P'_{(\log_z \alpha )^{-1},\delta}(A).
\end{eqnarray*}
Since $A= \{D \geq h\}$ is increasing
$$
P_{z,\alpha}(D \geq h) \leq P'_{\alpha, z, \sqrt{d}}(D \geq h)
\leq \bar P'_{(\log_z \alpha)^{-1},\sqrt{d}}(D \geq h).
$$
If we take $s=(\log_z \alpha)^{-1}$ and $\alpha \in (1, z^d)$ then 
$s \in (1/d, \infty)$ 
and the result follows from Proposition 3.1 with $\delta=\sqrt{d}$. 
\end{proof}

From lemma \ref{teo:lowerbound} and theorem \ref{teo:upperbound}, for large $h$ it
holds
$\PP_{z,\alpha}(D = h) \approx h^{-\gamma}$
where $\gamma-1 = \frac{\log_z \alpha}{d-\log_z \alpha}$. \\
Thus the hierarchical model is scale free 
for each $\alpha \in [1, z^d)$. Typically,
in the scale free region $-3 \leq -\gamma \leq -2$,
which is then equivalent to
$z^{\frac{d}{2}} \leq \alpha \leq z^{\frac{2d}{3}}$.

We end this section by commenting on the relation between the 
scale free region and the average number of communities to which
an individual belongs. As we have seen, there is a realistic 
average number of
communities for $\alpha \in [5/4,2]$, which has no
intersection with the typical scale-free region
even for  $d=2$ and $z=2$. It is, however, quite simple
to realign the parameter ranges  by introducing some more parameters
more realistically describing small group membership. This is reminiscent of
long range percolation in dimension 1, in which the
probability of nearest neighbor connection can, by
itself, determine phase transition for a critical value of 
the main parameter (\cite{AizenmanNewman}). We do not pursue this direction here.

\section{Epidemics}
We consider a Reed-Frost dynamics to describe the spread of an infection
on the connectivity network (see, for instance, \cite{BrittonDeijfenAndreasLagerasLindholm}, section 3, for 
a detailed description). In such dynamics, at discrete times each
infected individual contacts each one of its neighbors with
some probability, and if the neighbor is susceptible it becomes infected;
in the meantime the infected recovers.
Differently from usual, we assume, however,
 that the probability of infectious contact depends on the 
communities shared by the two neighbors: in particular, 
we assume that there is a probability of independent
transmission for each community shared by two individuals,
and we are interested in the set of individual eventually affected by the epidemics started
from one single vertex, the origin for instance. Such
set can be identified with the cluster $V_0^{(d)}$
containing the origin in an edge
percolation process on $G_{\alpha,z}$ 
described by the following probability measure:
for each value $x \in \mathbb R^d$
of the $X_v$'s, consider a (conditional) Bernoulli probability distribution
${\mathbb P}_{x, z, \rho, p}$ on
$\{0,1\}^{E_{x, z}}$ such that 
\begin{equation}\label{defK}
{\mathbb P}_{x,  z, \rho, p}(\eta_{\{u,v\}}=1)=1-\prod_{k=0}^{\infty} ( 1-p \rho^k \ \mathbb I_{\{\{u,v\}| \ \exists i \in \mathbb Z^d: \ x_u,x_v \geq k \textrm{ and } u,v \in B_{k,z}(i) \}}).
\end{equation}
Our main interest here is
in studying for which values of the parameters there is a finite
or an infinite set of infected individuals, or, equivalently,
a finite or infinite cluster, i.e. we are interested in 
the probability $\mathbb P_{x, z, \rho, p}
(| V^{(d)}_{0}|= \infty)$.
The joint probability distribution which describes percolation
and epidemics is defined on the Borel $\sigma$-algebra in $H$
by
$\mathbb P_{\alpha, z, \rho, p}=
\int_X \mathbb P_{x, z, \rho, p} \ \mu_{\alpha}(dx).
$

For a given $x$   let $G_{x,z}$ be the realization of the connectivity
graph with value $x$ of the $X_v$ variables. 
Since $G_{x,z}$ contains all the nearest neighbor edges and they are open with probability at least $p$, 
 if $p > \pi_c^{(d)}$, the critical point for  d-dimensional bond percolation, then
percolation occurs regardless of the value of the other parameters and of 
the realization $x$. Notice that $\pi_c^{(d)} <1 $ by Peierls argument, and,
more precisely, $\pi_c^{(2)}=1/2$ (\cite{Kesten1}) and 
$\pi_c^{(d)} \sim 1/2d$ (\cite{Kesten2}).
Moreover, for any fixed $x$ and $\rho$, the probability in \eqref{defK}
is increasing in $p$, and the random variables $\eta_{\{u,v\}}$ are
independent. Thus, it follows by a standard FKG inequality (see, e.g., 
\cite{Grimmett}) that for any $p \geq p'$ and any increasing event 
$A \subseteq \{0,1\}^{E_{x,z}}$ we have
${\mathbb P}_{x,  z, \rho, p}(A) \geq {\mathbb P}_{x,z, \rho, p'}(A)$.
Since $A^{(x)}_{0, \infty}=\{\eta: V^{(d)}_{0}|= \infty \}$ is increasing
it follows that there exists a critical $p_c(x, \rho) < 1$
for the onset of an infinite percolation cluster.

It could happen that $p_c(x, \rho)=0$. If 
$\alpha=1$ then we are assuming  $X_v=\infty$ for all $v$, 
and the percolation model is quite close to long range percolation
(\cite{NewmanSchulman}) in which the critical threshold
$\pi_c$ has a transition at some value of a parameter
which corresponds to $\rho$: for small values of $\rho$ we have $\pi_c>0$ and
for large $\rho$ it is instead $\pi_c=0$. After showing 
that, in fact, the critical threshold $p_c(x, \rho)=0$
is almost surely constant in $x$, we see that a similar 
transition occurs in the hierarchical model for all
values of $\alpha \in [1, z^d]$. To this purpose
we introduce a more detailed  description of the model: consider $\Sigma=\mathbb N^{\mathbb Z^d} \times
\{0,1\}^{ \mathbb E^d \times \mathbb N } $ and parameters $\alpha, \rho$ and $p$.
Then take a Bernoulli
 probability distribution 
$\tilde {\mathbb P}_{\alpha,  \rho, p}$
on the Borel $\sigma$-algebra $\mathcal A$ in $\Sigma$ such that
\begin{itemize}
\item $\tilde P_{\alpha, \rho,p} (\sigma_v \geq k) = \alpha^{-k}
\text{ for all } v \in \mathbb Z^d$;
\item $\tilde P_{\alpha, \rho,p} (\sigma_{\{u,v\},k} =1) = p \rho^k
\text{ for all } \{u,v\} \in \mathbb E^d \text{ and } k \in \mathbb N$
\end{itemize}
One then retrieves the probability $P_{\alpha, z, \rho,p}$ by 
considering the map $\psi_z: \Sigma \rightarrow H$
such that 
$$
\psi_z = \mathbb I_{ \{ \exists k \in \mathbb N, \ i \in \mathbb Z^d: \ u,v \in B_{z,k}(i); \ \sigma_u, \sigma_v \geq k; \ \sigma_{(\{u,v\}, k)}=1 \}}
$$
and observing that 
$P_{\alpha, z, \rho,p}=\tilde { P}_{\alpha,  \rho, p}
(\psi_z^{-1})$.
Notice that while $\tilde P_{\alpha,  \rho,p}$
is Bernoulli, the distribution 
$P_{\alpha, z, \rho,p}$ on $H$ is not independent since,
for instance, if $u,u' \in B_{z,k}(0) \setminus 
B_{z,k-1}(0)$ then 
$P_{\alpha, z, \rho,p}(\eta_{\{0,u'\}}=1|\eta_{\{0,u\}}=1)
= \alpha^{-k} \neq \alpha^{-2k} =
P_{\alpha, z, \rho,p}(\eta_{\{0,u'\}}=1)$.
$P_{\alpha, z, \rho,p}$ is actually one-dependent. xxxx

\begin{lem} 
$p_c(x, \rho) $ is almost everywhere constant in $x$
\end{lem}
\begin{proof}
Under 
$\tilde P_{\alpha,  \rho,p}$ the variables $\sigma_u$'s and $\sigma_{(\{u,v\},k)}$'s are
collectively independent.
Consider the $\sigma$-algebra $\bar {\mathcal A}_n$
generated by the variables with index in $\{(v, \{v,u\},k):
 v \in ([-n,n]^d \cap \mathbb Z^d)^c; \{u,v\} \in \mathbb E^d, u, v \in 
([-n,n]^d \cap \mathbb Z^d)^c; k \geq n\}$;
then ${\mathcal A}_{\infty} = \cap_n \bar {\mathcal A}_n$
is trivial under $\tilde P_{\alpha,  \rho,p}$. \\
Since the event $A_{\infty}= \{ \eta| \exists v \in \mathbb Z^d:
|V_v^{(d)}| = \infty \text{ in } \eta\}$
is  such that $\psi_z^{-1}(A_{\infty}) \in \bar {\mathcal A}_n$
for all $n \in \mathbb N$, then
$\psi_z^{-1}(A_{\infty}) \in  {\mathcal A}_{\infty}$
and $\tilde P_{\alpha,  \rho,p}(\psi_z^{-1}(A_{\infty}))=0,1$. 
Thus, $A_{\infty}$ has probability zero or one for
$P_{\alpha, z, \rho,p}$-a.a. $\eta \in H$. Hence,
$P_{x, z, \rho,p}(A_{\infty})=0,1$ for $\mu_{\alpha}$-a.a.
$x \in X$. Since $p_c(x, \rho)$
exists for all $x \in X$, it is $\mu_{\alpha}$ almost surely
constant.
\end{proof} 

We can define 
$p_c(\alpha, \rho) = \inf \{p: \tilde{\mathbb P}_{\alpha,  \rho, p}(\psi_z^{-1}(A_{\infty}))=1\}$. We already know that $p_c(\alpha, \rho) < 1$. 
We see now that $p_c(\alpha, \rho)=0$
when the transmission probabilities for large communities do not
decrease fast enough.

\begin{lem}
For $\alpha \in [1, z^d)$ and $\rho > \frac{\alpha}{z^d}$ we have 
$p_c=p_c(\alpha, \rho)=0$.
\end{lem}
\begin{proof}
The joint probability $\tilde {\mathbb P}_{\alpha,  \rho, p}$
suggests several dynamic constructions of the epidemics together 
with the reference graph; one  is the following. Starting from
the origin $0$ consider the sequence of boxes $B_{z,k}=B_{z,k}(0)$,
$k=1, \dots$ and sequentially generate the following  variables:
\begin{itemize}
\item[$(0)$] $\sigma_0$; 
\item[$(1_a)$] $\sigma_v, v \in B_{z,1}$; 
\item[$(1_b)$] $\sigma_{\{0,v\},1}, v \in B_{z,1}$; 
\item[\dots] 
\item[$(k_a)$]   $\sigma_v, v \in B_{z,k}\setminus B_{z,k-1}$;   
\item[$(k_b)$] $\sigma_{\{v,u\},j}, u,v \in B_{z,k-1} \setminus B_{z,k-2} ,j=1,\dots,k-1$;
\item[$(k_c)$] $\sigma_{\{v,u\},k}, u \in B_{z,k-1}, v \in B_{z,k} $; 
\item[\dots]
\item[(Last)] $\sigma_{\{u,v\},0}$ for all nearest neighbor pairs $\{u,v\}$.
\end{itemize}
Note that at every step only new $\sigma$ variables are generated,
that the last step can be performed at any time, 
possibly subdivided in several steps, and that the procedure
 generates all relevant
$\sigma$ variables in the positive orthant: in fact, 
if $v',u' \in B_{z,k}\setminus B_{z,k-1}$ then 
$\sigma_{\{v,u\},j}$ is generated at step $(((k+1)_b)$ for
$j=1, \dots, k$ and $(j_c)$ for all $j \geq k+1$; if,
instead, 
$v' \in B_{z,k} \setminus B_{z,k-1}$ and 
$v' \in B_{z,k+r}\setminus B_{z,k+r-1}$, $r \geq 1$, then 
$\sigma_{\{v,u\},j}$  for
$j=1, \dots, k+r-1$ is not generated but it is also not relevant in the process and 
for $j \geq k+r$ is generated at step $(j_c)$.

Following this construction we can show that for 
$\alpha \in [1, z^d)$, $\rho > \frac{\alpha}{z^d}$
and any $p>0$ there is an infinite cluster. We generate
a sequence $i_k , k \in \mathbb N$,
of vertices in $ B_{z,k} \setminus B_{z,k-1}$
or empty sets with the following procedure,
in which  the definition of $i_k$ depends on $3$ events which may occur
depending on the status of $i_{k-1}$:
\begin{itemize}
\item $\text{ if } \sigma_0 \geq 1 \textrm{ then } i_0=0 , \text{ else } i_0= \emptyset$;
\item  if $i_{k-1} \in B_{z,k-1} \setminus B_{z,k-2}$
and $\exists v \in B_{z,k} \setminus B_{z,k-1}$ such that $\sigma_v \geq k+1
\text{ and } \sigma_{\{i_{k-1},v\},k}=1$
then $i_{k}$ equals one of such vertices $v$ (the first in some fixed order); 
\item if $i_{k-1} \in B_{z,k-1} \setminus B_{z,k-2}$
and $\exists v \in B_{z,k} \setminus B_{z,k-1}: \sigma_v \geq k+1
\text{ but for all such $v$'s } \sigma_{\{i_{k-1},v\},k}=0$
then $i_{k}$ equals one of  vertices $v$  with the first two properties
(the first in some fixed order); 
\item if $i_{k-1} \in B_{z,k-1} \setminus B_{z,k-2}$
and for all $ v \in B_{z,k} \setminus B_{z,k-1}$ we have $\sigma_v < k+1$
then $i_{k}= \emptyset $; 
\item  if $i_{k-1} = \emptyset$
and $\exists v \in B_{z,k} \setminus B_{z,k-1}: \sigma_v \geq k+1$
then $i_{k}$ equals one of such vertices $v$ (the first in some fixed order);
\item if $i_{k-1} = \emptyset$
and for all $ v \in B_{z,k} \setminus B_{z,k-1}$ we have $\sigma_v < k+1$
then $i_{k}= \emptyset $; 
\end{itemize}
Given the vertices $i_k$'s we can define the events:
\begin{itemize}
\item $A_k= \{\exists v \in B_{z,k} \setminus B_{z,k-1}: \sigma_v \geq k+1, \sigma_{\{i_{k-1},v\},k}=1 \}$
\item $C_k = \{\exists v \in B_{z,k} \setminus B_{z,k-1}: \sigma_v \geq k+1
\text{ but either } i_{k-1} = \emptyset
\text{ or for all such $v$'s } \sigma_{\{i_{k-1},v\},k}=0 
 \}$
\item $E_k = \{\text{for all }  v \in B_{z,k} \setminus B_{z,k-1} 
\text{ it holds } \sigma_v < k+1
\}$
\end{itemize}
where clearly $A_k$ is not defined if $i_{k-1}= \emptyset$.
Notice that all the events $A_k, C_k$ and $E_k$ are
defined in terms of the variables at steps $(k_a)$ and
$(k_c)$ of the construction outlined above. This implies
that such events are defined in terms of  variables which, once
$i_{k-1}$ is given, are independent from
those involved in defining $A_i, C_i$ and $E_i$ for
$i=1, \dots, k-1$. 
Moreover, for each $k$ the three events form
a partition of the probability space. Therefore, 
the sequence $Z_k = a_k (c_k, e_k \text{ respectively })$
if $A_k (C_k, E_k \text{ respectively })$ occurs, is a (non-homogeneous) Markov chain,
whose transition matrix can be estimated in terms of the $\sigma$ variables.
In fact, 
\begin{eqnarray}
P(Z_k= a_k| Z_{k-1}=a_{k-1})&=& 1 - \Big(1 - \frac{p \rho^k}{\alpha^{k+1}}\Big)^{z^{d k}- z^{d(k-1)} } 
\geq 1 - e^{- \frac{p \rho^k(z^{d k}- z^{d(k-1)})}{\alpha^{k+1}}} \\
P(Z_k= c_k| Z_{k-1}=e_{k-1})&=& 1 - \Big(1 - \frac{1}{\alpha^{k+1}}\Big)^{z^{d k}- z^{d(k-1)} } 
\geq 1 - e^{- \frac{(z^{d k}- z^{d(k-1)})}{\alpha^{k+1}}}
\end{eqnarray}
and all other conditional probabilities are smaller than
$e^{- \frac{p \rho^k(z^{d k}- z^{d(k-1)})}{\alpha^{k+1}}}$ if
$Z_{k-1}=a_{k-1}$ or $Z_{k-1}=c_{k-1}$ and smaller than
$e^{- \frac{(z^{d k}- z^{d(k-1)})}{\alpha^{k+1}}}$ 
 if $Z_{k-1}=e_{k-1}$.
 
\noindent We have
$$
P(Z_k=e_k)=  \sum_{z=a_{k-1}, c_{k-1}, e_{k-1}}P(Z_k=e_k| Z_{k-1}=z) \ P(Z_{k-1}=z)
\ \leq \ e^{- \frac{p \rho^k(z^{d k}- z^{d(k-1)})}{\alpha^{k+1}}}
$$
and 
$$
P(Z_k=c_k) = \sum_{z=a_{k-1}, c_{k-1}, e_{k-1}}P(Z_k=c_k| Z_{k-1}=z) \ P(Z_{k-1}=z) 
\leq  2 \  e^{- \frac{p \rho^{k-1}(z^{d (k-1)}- z^{d(k-2)})}{\alpha^{k+1}}}
$$
so that if $\rho > \alpha/z^d$
$$
\sum_{k=1}^{\infty} P(Z_k=e_k)
< \infty, \quad \sum_{k=1}^{\infty} P(Z_k=c_k)
< \infty.
$$
By the first Borel-Cantelli Lemma
$E_k$ and $C_k$ occur only a finite number of times, so that
with probability one the
sequence terminates with one $C_k$ and then $A_h$ for $h>k$.
In such case the vertex $i_k$ is connected to an infinite cluster
containing all vertices $i_h$ for $h >k$. Since there are countably
many vertices there must be one $k$ and one vertex
$v \in B_{z,k} \setminus B_{z,k-1}$ which is starting vertex of 
an infinite cluster using edges in communities at level at least
$k$ with probability $c_1 >0$. Such vertex can be connected
to the origin using nearest neighbor edges, which are independent
from the previous construction as they were involved only in
the last step of the dynamic joint generation of graph and epidemic, 
with  some probability $c_2 >0$. In the end, the probability of percolation
from the origin is at least $c_1 c_2 >0$.
\end{proof}

\section{Domination by long-range percolation}
The description of the $\alpha-\rho$ phase diagram is completed
by the following result.
\begin{teo} \label{Positivepc}
For $\alpha > z^d$ or $\alpha \in [1, z^d]$ and
$\rho < \alpha/z^d$ we have $p_c >0$.
\end{teo}
This amounts to prove that, with the parameters $\alpha$ and $\rho$ in the
indicated region, there exists $p>0$ such that 
percolation does not occur for that value of $p>0$.
To show this, we actually bound the probability of existence of an infinite percolation cluster or infinite infected area
in the nested model with that in a long-range percolation, for which it
is easy to show that percolation does not occur for some values of the
parameter by bounding it with
a subcritical Galton-Watson process. \\
A long-range percolation model is defined as a probability on the Borel
$\sigma$-algebra in $H$ such that $Q_{\beta, s}(\eta_{\{u,v\}}=1)
=\frac{\beta}{(d(u,v))^s}$. 
\begin{teo} \label{DomLongRange}
When $s=\log_z(\alpha/\rho)$ and 
$\beta'=\frac{p}{1-\rho}(\frac{\alpha}{\rho})^{\frac{1}{2} \log_z d}$, it holds that
$$
\mathbb P_{\alpha, z, \rho,p} ( |V_0^{(d)}|=\infty)
=\tilde {\mathbb P}_{\alpha,  \rho, p} ( \psi_z^{-1}(|V_0^{(d)}|=\infty))
\leq Q_{\beta', s}(|V_0^{(d)}|=\infty).
$$
\end{teo}
The main difficulty lies in 
the fact that in the nested hierarchical model the distribution on the edges is
one dependent: we face this problem later on. Initially,  we once again compare
 the percolation network to 
$G'_{\alpha, z, \delta}$ endowed with 
slightly larger infection probabilities than
in the nested  model. \\
Let $u,v \in \mathbb Z^d$, define 
$$
k_{1,\delta}(u,v)=\lceil \log_z \frac{d(u,v)}{\delta} \rceil,
$$
and  
consider a Bernoulli
 probability distribution 
$\tilde {\mathbb P'}_{\alpha, z, \rho, p, \delta}$
on the Borel $\sigma$-algebra $\mathcal A$ in $\Sigma'
=\mathbb R^{\mathbb Z^d} \times\{0,1\}^{\mathbb E^d_n } $ such that
\begin{itemize}
\item $\tilde {\mathbb P'}_{\alpha, \rho,p, \delta} \ (\sigma'_v \geq k) = \alpha^{-k} \qquad \text{ for all } v \in \mathbb Z^d$;
\item $\tilde {\mathbb P'}_{\alpha, \rho,p, \delta} \ (\sigma'_{\{u,v\}} =1) = \frac{p}{1-\rho} \ \rho^{k_{1,\delta}(u,v)} \qquad 
\text{ for all } \{u,v\} \in \mathbb E^d.$
\end{itemize}
Consider then the map $\psi'_{z,\delta}: \Sigma' \rightarrow \mathbb E^d$
such that 
$$
(\psi'_{z,\delta}(\sigma'))_{\{u,v\}} = \mathbb I_{ \{
\sigma'_u, \sigma'_v \geq k_1(u,v); \ \sigma'_{\{u,v\}}=1) \}}
$$
\begin{lem} \label{Lemma1}
For all increasing events $A \subseteq H$, 
$\tilde {\mathbb P}_{\alpha,  \rho, p} ( \psi_{z,\delta}^{-1}(A))
 \leq \tilde {\mathbb P'}_{\alpha,z,  \rho, p, \delta} ((\psi'_{z,\delta})^{-1}(A))$
\end{lem}
\begin{proof}
Consider the $\sigma$-algebra $\mathcal A_X$ generated by the variables
$\sigma_u, u \in \mathbb Z^d$, and let $\tilde {\mathbb P}_{x , \rho, p} 
=\tilde {\mathbb P}_{\alpha,  \rho, p} (\cdot |x)$
and $\tilde {\mathbb P'}_{x , \rho, p, \delta}$ be the 
conditional probabilities of $\tilde {\mathbb P}_{\alpha,  \rho, p} $
and $\tilde {\mathbb P'}_{\alpha, z, \rho, p, \delta} $, respectively,
given $\mathcal A_X$. Notice that 
the conditional probabilities no longer depend on $\alpha$
and that  $\tilde {\mathbb P}_{x , \rho, p} (\psi_z^{-1})$
and $\tilde {\mathbb P'}_{x ,z, \rho, p, \delta} ((\psi'_{z,\delta})^{-1})$
are Bernoulli distributions on (the Borel $\sigma$-algebra of)
$H$ under which 
$$
\tilde {\mathbb P}_{x , \rho, p} \Big(\psi_{z}^{-1}(\sigma_{\{u,v\}}=1) \Big)
=1-\prod_{k \in I_x(u,v)}(1-  p \rho^k ) 
$$ 
where 
$$
I_x(u,v)=
\{k | \ \exists i \in \mathbb Z^d: u, v \in B_{k,z}(i)
\text{ and } x_u, x_v \geq k\},
$$
and
$$
\tilde {\mathbb P'}_{x , z,\rho, p, \delta}
 \ \Big((\psi'_{z,\delta})^{-1}(\sigma'_{\{u,v\}}=1) \Big)
= \frac{p}{1-\rho} \ \rho^{ \log_z \frac{d(u,v)}{\delta} }
$$
Note also that  $
I_x(u,v) \subset \{k_{1,\sqrt{d}}(u,v),k_{1,\sqrt{d}}(u,v)+1, \ldots ,\min(x_u,x_v) \}$
so that
\begin{eqnarray}
1-\prod_{k \in I_x(u,v)}(1- p \rho^k ) &\leq& 
1-\prod_{k \geq k_{1,\sqrt{d}}(u,v)}(1-  p \rho^k ) \nonumber \\
& = & p \sum_{k \geq {k_{1,\sqrt{d}}(u,v)}} \rho^k - p \sum_{h > k \geq {k_{1,\sqrt{d}}(u,v)}} \rho^{k+h} + \dots \nonumber \\
& \leq & \frac{p}{1-\rho} \ \rho^{ \log_z \frac{d(u,v)}{\sqrt{d}}} \nonumber
\end{eqnarray}
since the series in the second line is alternating with decreasing coefficients. Therefore, 
$$
\tilde {\mathbb P}_{x , \rho, p} (\psi_z^{-1}(\sigma_{\{u,v\}}=1))
\ \leq \
\tilde {\mathbb P'}_{x , \rho, p, \sqrt{d}} \ ((\psi'_z)^{-1}(\sigma_{\{u,v\}}=1))
$$
and $\tilde {\mathbb P'}_{x , \rho, p, \sqrt{d}} \ ((\psi'_z)^{-1})$
dominates in the FKG sense $
\tilde {\mathbb P}_{x , \rho, p} (\psi_z^{-1})$.

\noindent Therefore, if $A \subseteq H$ is increasing then 
$$
\tilde {\mathbb P}_{\alpha,  \rho, p} ( \psi_z^{-1}(A))
= \int_X \tilde {\mathbb P}_{x , \rho, p} (\psi_z^{-1} (A) ) 
\ \mu_{\alpha}(dx)
\leq \int_X \tilde {\mathbb P'}_{x ,z, \rho, p, \sqrt{d}} ((\psi'_{z,\sqrt{d}})^{-1}(A) )
\ \mu_{\alpha}(dx)
 = \tilde {\mathbb P'}_{\alpha,  \rho, p, \sqrt{d}} ((\psi'_z)^{-1}(A))
 $$
\end{proof}

To compare the percolation network $\tilde {\mathbb P'}_{\alpha, z, \rho, p, \sqrt{d}} ((\psi'_{z,\sqrt{d}})^{-1})$ with a long-range percolation network
we are going to prove an analogue of Theorem 3.1 in \cite{MeesterTrapman}.
In this direction there are two main problems. On one side, 
\cite{MeesterTrapman} applies to directed paths; on the other
side, connectivities in \cite{MeesterTrapman} are described by
convex functions $k(X_v,X_u)$ and for values of $X_u=x_u$ 
the connectivities are bounded
by expected values $\bar x_v= E(k(X_v, x_u))$. In that paper
the reason why the connections become independent in different
directions is that the 
$\bar x_v$'s are constant.

The directionality of the paths is easy to fix: paths under 
$\tilde P'_{\alpha,z, \rho, p}$ are not directed, but can trivially be
considered so by fixing an order along each path. Paths are
instead ordered under $\tilde { P'}_{\alpha,  \rho, p, \sqrt{d}} ((\psi'_z)^{-1})$
since the involved edge variables are defined according to an order.
Theorem 3.1 of \cite{MeesterTrapman} applies to hoppable collections
of paths, such as the collection of all self-avoiding paths starting at the
origin and reaching the boundary of some fixed set; since from each path one can extract a self-avoiding one, 
Theorem 3.1 applies to the occurrence of a connection from the origin to the
boundary as well.

As to the connectivity functions, the analogous in the present context would be
$k(X_v, X_u) = (\phi_z(X))_{\{v,u\}}$ which is not convex and cannot be easily
related to any constant value. To proceed, we introduce families of i.i.d. random
variables, one family for each $v \in \mathbb Z^d$, of the form $X''_{(v,u)}, 
u \in \mathbb Z^d \setminus \{v\}$, and then bound $\tilde {\mathbb P'}_{\alpha, z, \rho, p, \sqrt{d}} ((\psi'_z)^{-1})$ by a network based on the $X''_{(v,u)}$'s.
Connections in different directions are independent and depend only on distances,
thus the network based on $X''_{(v,u)}$'s is actually a long-range percolation model.
This is possible if we take the probability that $X''_{(v,u)}\geq k$
greater than or equal to the square root of the probability that 
$X'_{v}\geq k$. This, in turn, implies that in the long-range model
the presence of a vertex is equivalent, in distribution, to the fact
that $X'_{v}\geq k$ for one of its end-points, say the smallest in some fixed order.
While this implies that the probability that the infection travels
a self-avoiding path is larger in the long-range model, 
Theorem \ref{teo:MT2010} below shows the same inequality holds for the 
probability that at least one paths is travelled among those in a fixed
suitable collection.

Consider $\beta > 0$ and a Bernoulli
 probability distribution 
$\tilde {\mathbb P''}_{\beta, z, \rho, p, \delta}$
on the Borel $\sigma$-algebra $\mathcal A$ in $\Sigma''
=\mathbb N^{\mathbb Z^d \times \mathbb Z^d \setminus \{(i,i), i \in \mathbb Z^{d}\}} \times\{0,1\}^{ E^d } $ such that
\begin{itemize}
\item $\tilde { \mathbb P''}_{\beta, z, \rho, p, \delta} (\sigma''_{(u,v)} \geq k) = \beta^{-k}
\text{ for all } (u,v) \in \mathbb Z^d \times \mathbb Z^d \setminus \{(i,i), i \in \mathbb Z^{d}\}$;
\item $\tilde {\mathbb P''}_{\beta, z, \rho, p, \delta} (\sigma''_{\{u,v\}} =1) = \frac{p 
\rho^{k_{1,\delta}(u,v)} }{1-\rho}
\text{ for all } \{u,v\} \in \mathbb E^d.$
\end{itemize}
Consider then the map $\psi''_{z,\delta}: \Sigma'' \rightarrow H$
such that 
$$
\psi''_{z,\delta}(\sigma'')_{\{u,v\}} = \mathbb I_{ \{\sigma''_{(u,v)}, \ \sigma''_{(v,u)} \geq k_{1,\delta}(u,v); \ \sigma''_{\{u,v\}}=1 \}}
$$
and   let
$P''_{\beta, z, \rho, p, \delta}=\tilde {\mathbb P''}_{\beta, z, \rho, p, \delta}
((\psi''_{z,\delta})^{-1})$.
We denote by 
$\tilde {\mathbb P''}_{x'', z, \rho, p, \delta} ((\psi''_{z,\delta})^{-1} )$
the conditional probability given $x'' \in X'' =
\mathbb N^{\mathbb Z^d \times \mathbb Z^d \setminus \{(i,i), i \in \mathbb Z^{d}\}}$.
Note that in passing from $\tilde {\mathbb P'}_{\alpha, z, \rho, p, \delta}$
to $\tilde {\mathbb P''}_{\beta, z, \rho, p, \delta}$ we have changed the network mechanism and kept the same transmission rates.


We introduce an interpolation between 
$\tilde {\mathbb P'}_{\alpha, z, \rho, p, \delta}$
and $\tilde {\mathbb P''}_{\beta, z, \rho, p, \delta}$. To this
purpose we select an ordering of $\mathbb Z^d
=\{v_1, v_2, \dots \}$ and, for $h=0, 1, \dots$,
we 
consider the sequence of sets
$V(0)= \emptyset, \dots,
V(h)=\{v_1, \dots, v_h\}$. For later purposes we take the 
order such that $V(n^d) = B_n = [0,n-1]^d \cap \mathbb Z^d$.
Then  we 
take a sequence of Bernoulli distributions $\tilde P_h$
defined 
on the Borel $\sigma$-algebras $\mathcal A(h)$ of $\Sigma(h)
=\mathbb N^{\mathbb Z^d  \setminus V(h)}
\times \mathbb N^{V(h) \times \mathbb Z^d \setminus \{(i,i), i \in \mathbb Z^{d}\}}
 \times \{0,1\}^{\mathbb  E^d } $ 
 by
\begin{eqnarray*} 
\tilde P_h(\sigma_v \geq k) & = & \frac{1}{\alpha^k} \qquad v \in  \mathbb Z^d  \setminus V(h) \\
\tilde P_h(\sigma_{(v,u)} \geq k) & = & \frac{1}{\beta^k} \qquad v \in  V(h), u \in \mathbb Z^d \setminus v \\
\tilde P_h (\sigma_{\{v,u\}}=1) & = & \frac{p}{1-\rho} \ \rho^{k_{1,\delta}(u,v)}
\end{eqnarray*}
Furthermore, define the map $\psi_{z,h}:\Sigma(h) \rightarrow H$
given by
$$
(\psi_{z,\delta,h}(\sigma))_{\{u,v\}}= \mathbb I_{ \{\sigma_{t(v,u)} \geq k_{1,\delta}(u,v), \ \sigma_{t(u,v)} \geq k_{1,\delta}(u,v), \ \sigma_{\{u,v\}}=1 \}}
$$
where  $t(u,v) = u$  if $u \in \mathbb Z^d \setminus V(h)$
and $t(u,v) = (u,v)$ if $u \in V(h)$.
We have $\tilde P_0(\psi_{z,0}^{-1})= \tilde P'_{\alpha,z, \rho, p,\delta}
(( \psi'_{z})^{-1}).$

Fix now a box $B_n = [0,n-1]^d \cap \mathbb Z^d$ and consider the 
variables $\sigma|_{B_n}$, which are the $\sigma$'s restricted to 
$B_n$, i.e. to the index set
$\{(v), (v,u), \{v,u\}: v,u \in B_n\}$. For $v,u \in B_n$
and $h \leq n^d$,
$(\psi''_{z,\delta,h}(\sigma''))_{\{v,u\}}$
and $(\psi_{z,\delta}(\sigma))_{\{v,u\}}$
depend only from $\sigma''|_{B_n}$ and 
$\sigma|_{B_n}$, respectively. Therefore, 
$\tilde P_{n^d}(\psi^{-1}_{z, n^d})= \tilde { P''}_{\beta, z, \rho, p, \delta}
((\psi'')^{-1}_{z,\delta})$ by the definition of $\tilde P_h$.

Given a box $B_n \subseteq \mathbb Z^d$ and $v \in B_n$,
let $E_{v,n}=\{\{v,u\} : u \in B_n \cap \mathbb Z^d\}$ and  consider
now a pair of (possibly empty) sets $A, B \subseteq E_{v,n}$,
which in our case coincides with both $E'_v$ and $E^*_v$ of \cite{MeesterTrapman},
any $|A|$-dimensional vector $x=(x_1, \dots, x_{|A|}) \in (\mathbb R^+)^{|A|}$ and any
$|B|$-dimensional vector $y=(y_1, \dots, y_{|B|})\in (\mathbb R^+)^{|B|}$.
For a fixed $h$, the values $x$ and $y$ are interpreted as realizations of
$X_u$ if $u \in V(h)$ or $X_{(u,v)}$ if $u \notin V(h)$, respectively.

For $A \subseteq E_{v,n}$ we indicate by $Z_{A}$ the event 
$\{\eta: \eta_{\{v,u\}}=0 \text{ for all } \{v,u\} \in A\} \subseteq H $
that none of the edges of $A$ is open, and for any probability $P$ on 
$H$ we define the zero functions 
$z_v(P;n; A, B; x, y) = P(Z_A \cup Z_B)$
 as the probability
that either none of the edges of $A$ is open or none of the
edges of $B$ is open; for any pair of probabilities 
$P^{(a)}$ and $P^{(b)}$ denote by $z_v(P^{(a)},n) \leq z_v(P^{(b)},n) $
the fact that $z_v( P^{(a)};n; A, B; x, y) \leq 
z_v( P^{(b)};n; A, B; x, y)$ for all 
pairs of disjoint and possibly empty sets of endpoints $A, B \subseteq E_{v,n}$,
all $x \in \mathbb R^{|A|}$ and $y \in \mathbb R^{|B|}$.
The extension of Theorem 3.1 in \cite{MeesterTrapman} that we 
are going to prove uses the following inequality.
\begin{teo}\label{teoE1}
If $\beta^2 = \alpha$ then for all $n,h \in \mathbb N$ such that 
$v_h \in B_n$, $z_{v_h}(\tilde P_{h-1}
(\tilde \psi_{z,\delta, h-1}^{-1}),n) \geq z_{v_h}(\tilde P_{h}(\tilde \psi_{z,\delta, h}^{-1}),n)$.
\end{teo}
\begin{proof}
For  fixed $B_n \subset \mathbb Z^d $ and  $v=v_h \in B_n$, notice that the events
$\tilde \psi_{z,\delta, h}^{-1}(Z_A)$ and $ \tilde \psi_{z,\delta, h}^{-1}(Z_B)$  are
measurable with respect to the variables $\sigma_{v}$,
$\sigma_{(v,u)}$ and $\sigma_{\{v,u \}}$ which are indexed in the set
$Z_{v,n}=\{v\} \cup \{\{v,u\}, u \in B_n \setminus \{v\}\}
\cup \{(v,u), u \in B_n \setminus \{v\}\}$. Then let
$A, B \subseteq E_{v,n}$, disjoint, with $|A|=r$ and $|B|=m$,
and $x \in \mathbb R^{|A|}$ and $y \in \mathbb R^{|B|}$
be fixed; we identify each edge in $A$ or $B$ by
its endpont different from $v$.  We then let
\begin{equation}
A \cup B =(u_1, u_2, \ldots , u_{m+r})
\end{equation}
indicate the vertices wihch are endpoints (different from $v$)
 of edges in $A \cup B$, ordered according to the distance of the
endpoint from $v$, which is $d(v,u_i) \leq d(v,u_{i+1})$. We also indicate 
 $A=\{v_1, v_2, \ldots , v_r\}$ and $B=\{ w_1, w_2, \ldots , w_m\}$.
For simplicity of notation denote by $d_{u_i}=d(v,u_i)$ the distance from $v$ to $u_i$ and by $\alpha_{u_i}, \beta_{u_i}$ the following probabilities
\begin{eqnarray}
\alpha_{u_i} & = & \mu_{\alpha} \Big(X_v \geq \log_z \frac{d_{u_i}}{\sqrt{d}}\Big) 
= \alpha^{- \log_z(\frac{d_{u_1}}{\sqrt{d}})}\\ \nonumber
\beta_{u_i} & = & \mu_{\beta} \Big(X_{v,u_i} \geq \log_z \frac{d_{u_i}}{\sqrt{d}}\Big)  = \beta^{- \log_z(\frac{d_{u_1}}{\sqrt{d}})}
\end{eqnarray}
Thus $\alpha_{v_i}=(\beta_{v_i})^2$. Furthermore,
let $q_{u_i}=\frac{p \rho^{k_{1,\delta}(v, u_i)}}{1-\rho}$ and $\mathbb P_1=\tilde P_{h-1}$ and $\mathbb P_2=\tilde P_{h}$; we
 want to prove that
\begin{eqnarray} \label{DisegUnione}
\mathbb P_1(Z_{A} \cup Z_{B}) \geq \mathbb P_2(Z_{A} \cup Z_{B}).
\end{eqnarray} 
Let's proceed by induction on the cardinality of $A$ and $B$. Note that if $|A|= 0$ or $|B|= 0$ then $\mathbb P_1(Z_{A} \cup Z_{B}) = \mathbb P_2(Z_{A} \cup Z_{B}) = 1 $.

$(i)$ Suppose $A =\{u\}$, $B = \{w\}$. By symmetry we can assume that $d_w<d_u$; 
 then $\alpha_w > \alpha_u$ and $\beta_w > \beta_u$. We have  
\begin{eqnarray*}
\mathbb P_1(Z_{A} \cup Z_{B}) & = & 1-\mathbb P_1(Z_{A}^c \cap Z_{B}^c)\\
& = & 1-\alpha_u q_u q_w  \\
& & \\
\mathbb P_2(Z_{A} \cup Z_{B}) & = & 1-\mathbb P_2(Z_{A}^c \cap Z_{B}^c)\\
& = & 1-\mathbb P_2(Z_{A}^c) \ \mathbb P_2(Z_{B}^c)\\
& = & 1-\beta_u q_u \ \beta_w q_w
\end{eqnarray*}
Since $\beta_w > \beta_u$ then $\beta_u \beta_w > \beta_u^2 = \alpha_u$ and
$$ \mathbb P_1(Z_{A} \cup Z_{B}) \geq \mathbb P_2(Z_{A} \cup Z_{B}) \quad \textrm{if } |A|=|B|=1.$$
In particular,  equality holds if $d_u=d_w$. \\

$(ii)$ Now consider $\{ u_1, u_2, \ldots u_{m+r}\}= \{ v_1, v_2, \ldots v_r\} \cup \{ w_1, w_2, \ldots w_m\} = {A} \cup {B}$ such that $d_{u_1} \leq d_{u_2} \leq \ldots \leq d_{u_{m+r}}$. Note that for any probability $\PP$
$$\PP(Z_{A} \cup Z_{B})= \mathbb P(Z_{A}) + \mathbb P(Z_{B} ) - \mathbb P(Z_{A} \cap Z_{B})$$
As before, consider the probability of $Z_{A} \cup Z_{B}$. 
With respect to $\mathbb P_1$, if $X_v < \log_z \frac{d_{v_1}}{\delta}$ then $Z_{A}$ occurs. 
Instead, if $\log_z \frac{d_{v_i}}{\delta} \leq X_v < \log_z \frac{d_{v_{i+1}}}{\delta}$ then there exist $i$ connections 
in the basic graph and $Z_{A}$ occurs if at least one of them is open. Thus
\begin{eqnarray*} 
\PP_1(Z_{A}) & = & (1-\alpha_{v_1}) + \sum_{j=1}^{r-1}(\alpha_{v_j}-\alpha_{v_{j+1}}) 
 \prod_{i=1}^j (1-q_{v_i}) + \alpha_{v_n} \prod_{i=1}^r (1-q_{v_i}) \\
& & \\
\PP_1(Z_{B}) & = & (1-\alpha_{w_1}) + \sum_{j=1}^{m-1}(\alpha_{w_j}-\alpha_{w_{j+1}}) 
 \prod_{i=1}^j (1-q_{w_i}) + \alpha_{w_m} \prod_{i=1}^m (1-q_{w_i}) \\
& & \\
\PP_1(Z_{A} \cap Z_{B}) & = & (1-\alpha_{u_1}) + \sum_{j=1}^{r+m-1}(\alpha_{u_j}-\alpha_{u_{j+1}}) 
 \prod_{i=1}^j (1-q_{u_i}) + \alpha_{u_{n+m}} \prod_{i=1}^{m+r} (1-q_{u_i}) .
\end{eqnarray*}
With respect to $\mathbb P_2$, since edges are open independently of each other, we have
\begin{eqnarray*}
\mathbb P_2(Z_{A}) & = & \prod_{i=1}^r (1-\beta_{v_i} q_{v_i}) \\
& & \\
\mathbb P_2(Z_{B}) & = & \prod_{i=1}^m (1-\beta_{w_i} q_{w_i}) \\
& & \\
\mathbb P_2(Z_{A} \cap Z_{B}) & = & \prod_{i=1}^r (1-\beta_{v_i} q_{v_i}) \prod_{i=1}^m (1-\beta_{w_i} q_{w_i}) .
\end{eqnarray*}
We proceed by induction on $m+r$: we show that if
\eqref{DisegUnione} holds for $m+r-1$ then it holds also for
$m+r$. The vertex $u_{m+r}$ can be either in $A$ or in $B$ and
we assume with no loss of generality that $u_{m+r}=v_r \in A$.
Then we show that if
 $\mathbb P_1(Z_{A'} \cup Z_{B}) \geq \PP_2 (Z_{A'} \cup Z_{B})$ with $|A'|=r-1$, $|B|=m$ then $\mathbb P_1(Z_{A} \cup Z_{B}) \geq \PP_2 (Z_{A} \cup Z_{B})$ with
 $A=A' \cup \{v\}$ and thus $|A|=r$, $|B|=m$. This is equivalent to show that 
\begin{equation}\label{eqZ}
\mathbb P_1(Z_{A} \cup Z_{B}) - \PP_1(Z_{A'} \cup Z_{B} ) \geq  \mathbb P_2(Z_{A} \cup Z_{B}) - \mathbb P_2(Z_{A'} \cup Z_{B} ) .
\end{equation}
By elementary calculation it turns out that
$$\mathbb P_1(Z_{A} ) - \mathbb P_1(Z_{A'} ) = - \alpha_{v_n} q_{v_n} \prod_{i=1}^{r-1} (1-q_{v_i})$$
$$\mathbb P_1(Z_{A} \cap Z_{B} ) - \PP_1(Z_{A'} \cap Z_{B} ) = 
- \alpha_{v_r} q_{v_r} \prod_{i=1}^{r-1} (1-q_{v_i}) \prod_{j=1}^m (1-q_{w_j})$$
thus
$$ \mathbb P_1(Z_{A} \cup Z_{B}) - \mathbb P_1(Z_{A'} \cup Z_{B} ) = 
- \alpha_{v_r} q_{v_r} \prod_{i=1}^{r-1} (1-q_{v_i}) \bigg[1-\prod_{j=1}^m (1-q_{w_j}) \bigg]$$
Similarly, with respect to $\mathbb P_2$ we have
$$\mathbb P_2(Z_{A} ) - \mathbb P_2(Z_{A'}) = 
- \beta_{v_r} q_{v_r} \prod_{i=1}^{r-1} (1-\beta_{v_i} q_{v_i})$$

$$\mathbb P_2(Z_{A} \cap Z_{B} ) - \mathbb P_2(Z_{A'} \cap Z_{B} ) = 
- \beta_{v_r} q_{v_r} \prod_{i=1}^{r-1} (1-\beta_{v_i} q_{v_i}) \prod_{j=1}^m (1- \beta_{w_j} q_{w_j})$$
so that
$$ \mathbb P_2(Z_{A} \cup Z_{B}) - \mathbb P_2(Z_{A'} \cup Z_{B}) = 
- \beta_{v_r} q_{v_r} \prod_{i=1}^{r-1} (1-\beta_{v_i} q_{v_i}) \bigg[1-\prod_{j=1}^m (1-\beta_{w_j} q_{w_j}) \bigg]$$
In order to prove inequality (\ref{eqZ}) we must show 
$$ -\alpha_{v_r} q_{v_r} \prod_{i=1}^{r-1} (1-q_{v_i}) \bigg[1-\prod_{j=1}^m (1-q_{w_j}) \bigg] \geq  
-\beta_{v_r} q_{v_r} \prod_{i=1}^{r-1} (1-\beta_{v_i} q_{v_i}) \bigg[1-\prod_{j=1}^m (1-\beta_{w_j} q_{w_j}) \bigg]$$
Since $\alpha_{v_r} = \beta_{v_r}^2$, this is equivalent to  show that
$$ \beta_{v_r} \prod_{i=1}^{r-1} (1-q_{v_i}) \bigg[1-\prod_{j=1}^m (1-q_{w_j}) \bigg] \leq  
 \prod_{i=1}^{r-1} (1-\beta_{v_i} q_{v_i}) \bigg[1-\prod_{j=1}^m (1-\beta_{w_j} q_{w_j}) \bigg]$$
Since $\beta_{w_i} \leq 1$, $1-q_{w_i} \leq 1-\beta_{w_i} q_{w_i}$, thus 
\begin{equation}\label{eq40}
\prod_{i=1}^{r-1} (1-q_{v_i}) \leq \prod_{i=1}^{r-1} (1-\beta_{v_i} q_{v_i})
\end{equation}
Moreover, we see now that
\begin{equation}\label{eq41}
\beta_{v_r} \bigg[1-\prod_{j=1}^m (1-q_{w_j}) \bigg] \leq  \bigg[1-\prod_{j=1}^m (1-\beta_{w_j} q_{w_j}) \bigg]
\end{equation}
proceeding by induction on $m$. If $m=1$ then
$$\beta_{v_r} q_{w}  \leq  \beta_{w} q_{w} $$
because $v_r$ is the vertex at maximal distance from $u$, so that $\beta_{v_r} \leq  \beta_{w} $. Next we evaluate the increment between the $(m-1)$-th and $m$-th term.
\begin{eqnarray*}
\beta_{v_r} \bigg[1-\prod_{j=1}^m (1-q_{w_j}) \bigg] -\beta_{v_r} \bigg[1-\prod_{j=1}^{m-1} (1-q_{w_j}) \bigg] & = & 
\beta_{v_r} q_{w_m} \prod_{j=1}^{m-1} (1-q_{w_j}) \\
\bigg[1-\prod_{j=1}^m (1-\beta_{w_j} q_{w_j}) \bigg] - \bigg[1-\prod_{j=1}^{m-1} (1-\beta_{w_j} q_{w_j}) \bigg] & = & 
\beta_{w_m} q_{w_m} \prod_{j=1}^{m-1} (1-\beta_{w_j} q_{w_j}) 
\end{eqnarray*}
thus inequality (\ref{eq41}) follows from inequality (\ref{eq40}) and $\beta_{v_n} \leq  \beta_{w_m}$.
\end{proof}

Now we are able to follow Meester and Trapman's work \cite{MeesterTrapman} to  bound from above the probability of large outbreak, i.e. the existence of an infinite open path, by the corresponding quantity in the long-range model. 
In order to prove the results below we need to recall some definitions; the
detailed definitions are in \cite{MeesterTrapman}.
An ordered set of edges in some $E \subseteq \mathbb Z^d \times \mathbb Z^d$
of the form $\xi= (v_0 v_1, v_1 v_2, \dots, v_{n-1} v_n)$ is a (directed) path
from $v_0$ to $v_n$.  A path 
$\xi= (v_0 v_1, v_1 v_2, \dots, v_{n-1} v_n, \dots)$ with infinitely many different
edges is an infinite path. Given a finite or infinite path $\xi= (v_0 v_1, v_1 v_2, \dots, v_{n-1} v_n)$ we indicate the truncation after $k$ edges
as $\xi^s(k)= (v_0 v_1, v_1 v_2, \dots, v_{k-1} v_k)$ and the tail 
starting after $k$ edges as $\xi^t(k)= (v_k v_{k+1},  \dots)$; for 
two paths $\xi_1= (v_0 v_1, v_1 v_2, \dots, v_{n-1} v_n)$
and $\xi_2= (v_n v_{n+1},  \dots)$ we denote the conjunction by
$(\xi_1,\xi_2)= (v_0 v_1, v_1 v_2, \dots, v_{n-1} v_n, v_n v_{n+1},  \dots)$.
Next, let $\Xi$ be a collection of paths; if $E^{(n)}$ is
the collection of the first $n$ edges of $E$ according
to some given enumeration of $E$ then we indicate by $\Xi_n$ the set
of finite paths of $\Xi$ all of which edges are in $E^{(n)}$
together with all the infinite paths of $\Xi$ truncated 
at the first instance they leave $E^{(n)}$.

Furthermore, given a configuration $\eta \in H = \{0,1\}^E$
we say that $\xi$ is open in $\eta$ if  for all edges $\{v_k,v_{k+1}\}$
we have $\eta_{\{v_k,v_{k+1}\}}=1$. And we indicate by
$ C^{\Xi}$ the event that at least one path in $\Xi$ is open.
 We say that $\Xi$ is {\it  hoppable} if 
\begin{itemize}
\item for any $v \in \mathbb Z^d$ and any two paths $\xi $ and $\phi$ 
of $\Xi$
going through $v$, where $v$ is the end vertex of the $i$-th edge of 
$\xi$ and the starting vertex of the $j$-th edge of $\phi$, then
$(\xi^s(i), \phi^t(j))\in \Xi$.
\item $\lim_n  C^{\Xi_n} =  C^{\Xi}$
\end{itemize}

\begin{teo} \label{teo:MT2010}
For every hoppable collection of paths $\Xi$ in $\mathbb E^d$ 
\begin{equation}
 \tilde { P'}_{\alpha, z, \rho, p,\delta} ((\psi'_{z,\delta})^{-1}(C^{\Xi}))
 \leq \tilde P''_{\sqrt{\alpha} ,z, \rho, p,\delta}((\psi''_{z,\delta})^{-1}(C^{\Xi}))
\end{equation}
\end{teo} 
\begin{proof}
We mimic the proof of  Theorem 3.1 of \cite{MeesterTrapman},
dividing the argument into 3 steps.
Since $\tilde { P'}_{\alpha, z, \rho, p,\delta}$ and
$\tilde P''_{\sqrt{\alpha} ,z, \rho, p, \delta}$ are not defined on the same space,
we use the interpolating distributions $\tilde P_{h}$,
which are such that two consecutive ones differ only 
in the variables related to a single vertex.
Fix a box $B_n
=[-n,n]^d \cap \mathbb Z^d$. 

$(i)$ The first step is  to show that for all $n$
and $h$ such that $v_h \in B_n$, 
$\tilde P_{h-1}(\psi^{-1}_{z, \delta,h-1}(C^{\Xi_n})) 
\leq \tilde P_{h}(\psi^{-1}_{z, \delta,h}(C^{\Xi_n}))$.

Since $\beta^2=\alpha$, by Theorem \ref{teoE1}, $z_{v_h}(\tilde P_{h-1},n) \leq z_{v_h}(\tilde P_{h},n)$.
Denote by 
$\Sigma'(h)
=\mathbb N^{\mathbb Z^d  \setminus V(h)}
\times \mathbb N^{(V(h)\setminus v_h) \times \mathbb Z^d \setminus \{(i,i), i \in \mathbb Z^{d}\}}
 \times \{0,1\}^{\mathbb E^d \setminus E_{v,n}}
 =\mathbb N^{\mathbb Z^d  \setminus (V(h) \cup v_h)}
\times \mathbb N^{(V(h-1)) \times \mathbb Z^d \setminus \{(i,i), i \in \mathbb Z^{d}\}}
 \times \{0,1\}^{\mathbb E^d \setminus E_{v,n}} $;
 by $\Sigma_n'(h)$
its restriction to $B_n$, and $\mathcal A'_h$ and $\mathcal A'_{h,n}$
the Borel $\sigma$-algebras generated by the variables in 
$\Sigma'(h)$ and $\Sigma_n'(h)$ respectively.
For all $h$
\begin{eqnarray}
\tilde P_{h} (\psi^{-1}_{z, \delta,h}(C^{\Xi_n}))&=& \int_{\Sigma_n'(h)}
\tilde P_{h} (\psi^{-1}_{z, \delta,h}(C^{\Xi_n})| \sigma'_{\Sigma_n'(h)})d \tilde P_{h} ( \sigma'_{\Sigma_n'(h)})\nonumber \\
&=& \int_{\Sigma_n'(h)}
\tilde P_{h,} (\psi^{-1}_{z, \delta,h}(C^{\Xi_n})| \sigma'_{\Sigma_n'(h)})d \tilde P_{h-1} ( \sigma'_{\Sigma_n'(h)})\nonumber, 
\end{eqnarray}
where for $\sigma' \in \Sigma_n'(h)$, $P( \quad |\sigma')$ is the conditional
probability given $\mathcal A'_{h,n}$; the last equality holds since
$\tilde P_{h}$ coincides with $\tilde P_{h-1}$ on $\mathcal A'_{h,n}$.
Therefore, 
\begin{eqnarray}
\tilde P_{h} (\psi^{-1}_{z, \delta,h}(C^{\Xi_n}))&-&
\tilde P_{h-1} (\psi^{-1}_{z, \delta,h-1}(C^{\Xi_n}))\nonumber \\
&=& \int_{\Sigma_n'(h)}
(\tilde P_{h} (\psi^{-1}_{z, \delta,h}(C^{\Xi_n})| \sigma'_{\Sigma_n'(h)}) -
\tilde P_{h-1} (\psi^{-1}_{z, \delta,h-1}(C^{\Xi_n})| \sigma'_{\Sigma_n'(h)})) d \tilde P_{h-1}( \sigma'_{\Sigma_n'(h)}). \nonumber 
\end{eqnarray}
Now one can follow the proof of Theorem 3.1 in \cite{MeesterTrapman}: if 
the event $C^{\Xi_n}$ occurs  in $\sigma'_{\Sigma_n'(h)}$ regardless of the variables
in $Z_{V_h,n}$,  then the integrand is $0$. Otherwise, one can follow verbatim case 3. of the proof of
Theorem 3.1 in \cite{MeesterTrapman} to conclude that 
$\tilde P_{h-1}(\psi^{-1}_{z, \delta,h-1}(C^{\Xi_n}|\sigma'_{\Sigma_n'(h)})) 
\leq \tilde P_{h}(\psi^{-1}_{z, \delta,h}(C^{\Xi_n}|\sigma'_{\Sigma_n'(h)}))$ for all $h=0, \dots, n^d-1$
and thus the unconditional inequality holds.

 $(ii)$ By iteration, 
$$
 \tilde { P'}_{\alpha, z, \rho, p,\delta} ((\psi'_{z,\delta})^{-1}(C^{\Xi_n}))
= \tilde P_{0}(\psi^{-1}_{z,\delta, 0}(C^{\Xi_n}))
\leq \tilde P_{n^d}(\psi^{-1}_{z,\delta, n^d}(C^{\Xi_n}))
=\tilde { P''}_{\alpha, z, \rho, p,\delta} ((\psi''_{z,\delta})^{-1}(C^{\Xi_n})).
$$

 $(iii)$ In the last step we consider a general hoppable collection of paths $\Xi$. By definition of hoppable collection of paths, since $C^{\Xi_n}$ is 
 decreasing in $n$, it follows that
\begin{eqnarray*}
\tilde { P'}_{\alpha, z, \rho, p,\delta} ((\psi'_{z,\delta})^{-1}(C^{\Xi}))
& = & \lim_{n \rightarrow \infty} 
\tilde { P'}_{\alpha, z, \rho, p,\delta} ((\psi'_{z,\delta})^{-1}(C^{\Xi_n}))\\
 \tilde { P''}_{\alpha, z, \rho, p,\delta} ((\psi''_{z,\delta})^{-1}(C^{\Xi}))& = & \lim_{n \rightarrow \infty}  \tilde { P''}_{\alpha, z, \rho, p,\delta} ((\psi''_{z,\delta})^{-1}(C^{\Xi_n}))
\end{eqnarray*}
and using the previous steps the proof is completed.
\end{proof}

\begin{proof} (of Theorem \ref{DomLongRange}).
For all hoppable collections of paths $\Xi$,
 $C^{\Xi}$ is an increasing event in $H$; moreover,
 $\{|V_0^{(d)}|=\infty\}
= C^{\Xi}$ when $\Xi$ is the collection of all infinite paths
containing the origin. 
If $s=\log_z(\alpha/\rho)$ and 
$\beta'=\frac{p}{1-\rho}(\frac{\alpha}{\rho})^{\frac{1}{2} \log_z d}$
then  
\begin{eqnarray}
\tilde P''_{\sqrt{\alpha},z, \rho, p,\delta}((\psi''_{z,\delta})^{-1}(\eta_{\{u,v\}}=1))
&=& \tilde P''_{\sqrt{\alpha},z, \rho, p,\delta}(\sigma_{(v,u)} \geq k_{1,\delta}(u,v), \ \sigma_{(u,v)} \geq k_{1,\delta}(u,v), \ \sigma''_{\{v,u\}}=1) \nonumber \\
&=& (\sqrt{\alpha})^{-2 k_{1,\delta}(u,v)} \frac{p 
\rho^{k_{1,\delta}(u,v)} }{1-\rho}
\nonumber \\
&=& \frac{p}{1-\rho} \Big(\frac{p}{\alpha} \Big)^{\lceil \log_z \frac{d(u,v)}{\sqrt{d}} \rceil}
\nonumber \\
&\leq& \frac{p}{1-\rho} \ \Big( \frac{\alpha}{\rho} \Big)^{\frac{\log_z{d}}{2}} \ 
d(u,v)^{-\log_z (\frac{\alpha}{\rho})}\nonumber \\
&=&  \frac{\beta'}{(d(u,v))^s}
=Q_{\beta', s}(\eta_{\{u,v\}}=1)
\nonumber
\end{eqnarray}
for a long-range percolation model $Q_{\beta', s}$.
Combining Lemma \ref{Lemma1} and Theorem \ref{teo:MT2010},
we have 
\begin{eqnarray}
P_{\alpha,z,\rho,p}(|V_0^{(d)}|=\infty) &=&
\tilde {\mathbb P}_{\alpha,  \rho, p} ( \psi_z^{-1}(|V_0^{(d)}|=\infty))
\nonumber \\
 &\leq & \tilde {\mathbb P'}_{\alpha, z, \rho, p,\delta} ((\psi'_{z,\delta})^{-1}(|V_0^{(d)}|=\infty))
 \nonumber \\
 &\leq & \tilde 
 P''_{\sqrt{\alpha},z, \rho, p,\delta}
 ((\psi''_{z,\delta})^{-1}(|V_0^{(d)}|=\infty))\nonumber \\
 &\leq & Q_{\beta', s}(|V_0^{(d)}|=\infty) \nonumber
\end{eqnarray}
\end{proof}

\begin{proof} (of Theorem \ref{Positivepc})
In order to establish for which values of the parameters $\alpha,p,\rho,z $ no percolation occurs, it's now sufficient to dominate the long-range percolation 
model $Q_{\beta', s}$ by a subcritical Galton Watson tree. Recall that a GW tree is subcritical, i.e. the probability of extinction is one, if the expected value of the descendants of any vertex is less or equals to one. If $R_v$ denotes
the number of neighbors of a vertex $v$ we have
\begin{eqnarray}
 E_{Q_{\beta', s}}(R_v) = 
 2 d p + \sum_{u \in \mathbb Z^d} \frac{p}{1-\rho} (\frac{\alpha}{\rho})^{\frac{\log_z{d}}{2}}
\frac{1}{d(u,v)^{ \log_z (\frac{\alpha}{\rho})}}
\leq 2 d p + \sum_{k \in \mathbb N} 2 d k^{d-1}
\frac{p}{1-\rho} (\frac{\alpha}{\rho})^{\frac{\log_z{d}}{2}}
\frac{1}{k^{ \log_z (\frac{\alpha}{\rho})}}
< \infty \nonumber
\end{eqnarray}
for all $\rho \in [0,1]$ if $\alpha > z^d$ or for 
$\alpha \in [1,z^d]$ and $\rho < \frac{\alpha}{z^d}$.
\end{proof}

\begin{figure}[t]
\begin{center} 
\includegraphics[height=4in,width=5in]{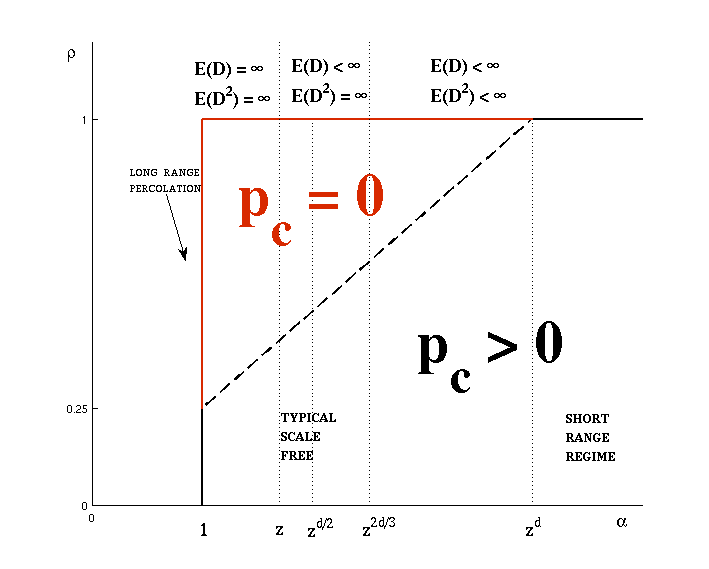}
\caption{\small \sl The phase space of the nested model in the  $\alpha-\rho$ 
plane.\label{fig:sfree}} 
\end{center}
\end{figure}

\end{document}